\documentclass[10pt,a4paper,reqno]{amsart}
\usepackage{amssymb,amsmath,amsthm,amstext,amsfonts}

\usepackage{enumitem}

\usepackage[most]{tcolorbox}

\usepackage{comment} % para poder comentar dentro de caixas com \begin{comment} comentário \end{comment}

\usepackage{hyperref}
\RequirePackage[dvipsnames]{xcolor} % [dvipsnames]
\definecolor{halfgray}{gray}{0.55} 
\definecolor{webgreen}{rgb}{0,0.5,0}
\definecolor{webbrown}{rgb}{.6,0,0} \hypersetup{%
	colorlinks=true, linktocpage=true, pdfstartpage=3,
	pdfstartview=FitV,%
	breaklinks=true, pdfpagemode=UseNone, pageanchor=true,
	pdfpagemode=UseOutlines,%
	plainpages=false, bookmarksnumbered, bookmarksopen=true,
	bookmarksopenlevel=1,%
	hypertexnames=true,
	pdfhighlight=/O,%hyperfootnotes=true,%nesting=true,%frenchlinks,%
	urlcolor=webbrown, linkcolor=RoyalBlue,
	citecolor=webgreen, %pagecolor=RoyalBlue,%
	pdftitle={},%
	pdfauthor={},%
	pdfsubject={2000 MAthematical Subject Classification: Primary:},%
	pdfkeywords={},%
	pdfcreator={pdfLaTeX},%
	pdfproducer={LaTeX with hyperref}%
}

\newtheorem{theorem}{Theorem}[section]
\newtheorem{lemma}[theorem]{Lemma}

\theoremstyle{definition}

\newtheorem{example}[theorem]{Example}

\numberwithin{equation}{section}
\theoremstyle{plain}

\numberwithin{equation}{section} %% Comment out for sequentially-numbered
\numberwithin{figure}{section} %% Comment out for sequentially-numbered
\theoremstyle{plain}
\theoremstyle{plain}
\theoremstyle{remark}
\newtheorem*{acknowledgement*}{Acknowledgement}

\DeclareMathOperator{\Id}{Id}

\newcommand{\cA}{{\mathcal A}}

\newcommand{\R}{{\mathbb R}}
\newcommand{\Z}{{\mathbb Z}}
\newcommand{\N}{{\mathbb N}}

\def\Reg{\mathcal{R}_\mu}

\usepackage{orcidlink}

\begin{document}
\title[Periodic approximation of LE for cocycles with holonomies]{Periodic approximation of Lyapunov exponents for cocycles admitting invariant holonomies}

\author[Lucas Backes]{Lucas Backes \orcidlink{0000-0003-3275-1311}}
\author[Breno Rilho Lemos]{Breno Rilho Lemos \orcidlink{0009-0000-8370-916X}}
\author[Breno Rocha]{Breno Rocha \orcidlink{0009-0009-5852-2753}}

\address{\noindent Departamento de Matem\'atica, Universidade Federal do Rio Grande do Sul, Av. Bento Gon\c{c}alves 9500, CEP 91509-900, Porto Alegre, RS, Brazil.}
\email{lucas.backes@ufrgs.br} 
\email{rilho.lemos@ufrgs.br}
\email{breno.rocha@ufrgs.br}


\begin{abstract}
Classical results establish that the Lyapunov exponents of an ergodic measure for linear cocycles over hyperbolic systems can be approximated by the Lyapunov exponents of periodic orbits, provided the cocycle is H\"older continuous. A recent counterexample by Bochi demonstrates that this approximation property fails in general if the H\"older assumption is relaxed to mere continuity. In this paper, we introduce a geometric condition that successfully substitutes this analytical regularity hypothesis. More precisely, we prove that if a cocycle - even a discontinuous one - admits a continuous family of invariant holonomies, the periodic approximation of Lyapunov exponents still holds.
Our geometric approach yields a proof that is substantially simpler and more direct than existing arguments in the literature, even when applied to classical settings such as fiber-bunched cocycles for which previous results were already available.
\end{abstract}

\maketitle
\section{Introduction}

The study of Lyapunov exponents for linear cocycles over hyperbolic dynamical systems is a central theme in smooth ergodic theory. A classical and widely investigated problem in this setting is the approximation of Lyapunov exponents associated with an ergodic measure by the Lyapunov exponents of periodic orbits. There is a rich and well-established literature showing that this approximation property holds under suitable conditions. We refer the reader to the works \cite{Bac, BD, Kal, KS2} and the references therein. However, a crucial hypothesis underlying all of these approximation results is that the cocycle possesses a certain degree of analytical regularity. Specifically, they all assume that the cocycle is H\"older continuous.

It is natural to ask whether this H\"older continuity assumption is strictly necessary, or if mere continuity of the cocycle is sufficient to guarantee the periodic approximation of Lyapunov exponents. Recently, Bochi \cite{Boc26} answered this question in the negative. By means of an explicit counterexample, he demonstrated that if the H\"older continuity assumption is removed, the approximation property fails in general. Consequently, to extend approximation results to the merely continuous setting, one must introduce an alternative mechanism to control the behavior of the cocycle over periodic orbits.

The aim of this note is to exhibit a broad class of continuous cocycles, which are not necessarily H\"older continuous, for which the periodic approximation property still holds. We achieve this by replacing the standard analytical regularity hypothesis with a natural geometric condition. More precisely, we prove that if a continuous cocycle admits a continuous family of invariant holonomies, then the Lyapunov exponents of any ergodic measure can indeed be approximated by the Lyapunov exponents of periodic points. In fact, our proof does not strictly rely on the continuity of the cocycle itself: it remains valid even for discontinuous cocycles, provided they admit continuous holonomies on a suitable open set.

The geometric condition we impose is highly natural, as invariant holonomies play a fundamental role in the study of cocycles over hyperbolic systems. For instance, continuous holonomies are a crucial tool in the study of the cohomological equation and rigidity phenomena \cite{Bac15, KS1}. Furthermore, the existence of such holonomies is a central ingredient in establishing the continuity of Lyapunov exponents \cite{BBB}. By exploiting this intrinsic geometric structure, we are able to bypass the local distortion estimates that typically necessitate H\"older bounds.

Finally, we emphasize that our geometric approach presents some advantages beyond the relaxation of the regularity hypothesis. Even in classical situations where the previous approximation results already apply, such as in the context of fiber-bunched cocycles, which are H\"older continuous and naturally admit invariant holonomies, the proof we present here is substantially simpler and more direct than previous arguments in the literature.

\section{Statements}
Let $(M,d)$ be a compact metric space and $f\colon M\to M$ be a homeomorphism. 

\subsection{Local product structure}
Given any $x\in M$ and $\varepsilon >0$, define the \emph{local stable} and \emph{local unstable sets}, respectively, by
\begin{align*}
   W^s_\varepsilon(x) &:= \left\{y\in M : d(f^n(x),f^n(y))\le\varepsilon,\ \forall n \ge 0\right\}
\end{align*}
and
\begin{align*}
   W^u_\varepsilon(x) &:= \left\{y\in M : d(f^n(x),f^n(y))\le\varepsilon,\ \forall n \le 0\right\}.
\end{align*}
We say that $f$ has \emph{continuous local product structure} if there exist constants $\varepsilon, \tau>0$ such that whenever $d(x,y)\le\tau$, the sets $W^s_\varepsilon(x)$ and $W^u_\varepsilon(y)$ intersect in a unique point, denoted by $[x,y]$, which depends continuously on $x$ and $y$.

\subsection{Periodic closing property} 
The map $f\colon M\to M$ is said to have the \emph{periodic closing property} if for every $\delta>0$ there exists $\varepsilon_0>0$ such that for any $x\in M$ and $n\in \N$ satisfying $d(x,f^n(x))<\varepsilon_0$, there exists a periodic point $p\in M$ with $f^n(p)=p$ such that 
 \[d(f^i(x),f^i(p))\le \delta \quad \text{for every } i=0,1,2,\ldots,n. \]

Notice that shifts of finite type and basic pieces of Axiom A diffeomorphisms are particular examples of homeomorphisms with local product structure exhibiting the periodic closing property (see for instance \cite[Chapter IV, \S~9]{ManeBook} for details).

\subsection{Cocycles and invariant holonomies} \label{sec: cocycles and inv hol}
Given a continuous map $A\colon M\to GL(d,\R)$, the \emph{cocycle} generated by $A$ over $f$ is defined as the map $\cA\colon \Z\times M\to GL(d,\R)$ given by
\begin{equation*}\label{def:cocycles}
A^n(x):=\cA(n, x)=
\begin{cases}
	A(f^{n-1}(x))\cdots A(f(x))A(x)  & \text{if } n>0, \\
	\Id & \text{if } n=0,\\
	A(f^{n}(x))^{-1}\cdots A(f^{-1}(x))^{-1} & \text{if } n<0
\end{cases}
\end{equation*}
for all $x\in M$. In what follows, we will refer to the cocycle generated by $A$ over $f$ simply as the cocycle $A$.

We say that $A$ admits a family of \emph{(continuous) invariant stable and unstable holonomies} if there exist (continuous) maps
\[H^s\colon \{(x,y) \in M \times M : y \in W^s_\varepsilon(x)\} \to GL(d, \R)\]
and
\[H^u\colon \{(x,y) \in M \times M : y \in W^u_\varepsilon(x)\} \to GL(d, \R)\]
satisfying
\begin{itemize}
    \item[i)]  $H^s_{x, x} = \Id$ and $H^u_{x, x} = \Id$;
    \item[ii)] $H^{s}_{y, z}=H^{s}_{x,z}H^{s}_{y,x}$ and $H^{u}_{y, z}=H^{u}_{x,z}H^{u}_{y,x}$;
    \item[iii)] $H^s_{f(x) , f(y)} A(x) = A(y)  H^s_{x , y} \quad\text{and}\quad H^u_{f(x) , f(y)} A(x) = A(y)  H^u_{x , y}$,
\end{itemize}
whenever these  holonomies are defined, and for $x,y$ and $z$ in the appropriate stable or unstable sets. In Section \ref{sec: examples} we present some examples of cocycles admitting such a family of invariant holonomies.

\subsection{Lyapunov exponents} 
Given an ergodic $f$-invariant Borel probability measure $\mu$, it follows by Oseledets' Theorem \cite[Theorem 4.2]{Via14} that there exists a full $\mu$-measure set $\Reg \subset M$, whose points are called $\mu$-regular points, such that for every $x\in \Reg$ there exist numbers $\lambda _1(\mu)>\lambda_2(\mu) > \cdots > \lambda_{\ell}(\mu)$, called \emph{Lyapunov exponents}, and a direct sum decomposition $\R^d=E^{1}_{x}\oplus E^{2}_{x}\oplus \ldots \oplus E^{\ell}_{x}$ into vector subspaces (called \emph{Oseledets subspaces}) which depend measurably on $x$ such that, for every $1\le i \le \ell$,
\begin{itemize}
\item $A(x)E^{i}_{x}= E^{i}_{f(x)}$;
\item $m_i:=\dim (E^{i}_{x})$ is constant and is said to be the \emph{multiplicity} of $\lambda_i(\mu)$;
\end{itemize}
and 
\begin{itemize}
    \item for every nonzero $v\in E^{i}_{x}$, 
\[\lambda _i (\mu) =\lim _{n\to \pm \infty} \dfrac{1}{n}\log \| A^n(x)v\|.\]
\end{itemize}
We denote by 
\[\gamma _1(\mu)\ge \gamma _2(\mu)\ge \ldots \ge \gamma _d(\mu)\]
the Lyapunov exponents of $(A,f,\mu)$ repeated according to multiplicities. 

\subsection{Main result} 
Recall that if $p\in M$ is a periodic point of period $n$, then
\[\mu_p=\frac{1}{n}\sum_{j=0}^{n-1}\delta_{f^j(p)}\]
is an ergodic $f$-invariant probability measure supported on the orbit of $p$. Then, the main result of this note is the following.

\begin{theorem}\label{theo: main} 
Let $f\colon M\to M $ be a homeomorphism with local product structure exhibiting the periodic closing property, $\mu$ an ergodic $f$-invariant probability measure, and $A\colon M\to GL(d,\R)$ a continuous map admitting a family of continuous invariant stable and unstable holonomies. Then, for any $\theta>0$, there exists a periodic point $p\in M$ such that
\begin{equation}\label{eq: main theo}
	|\gamma _i (\mu)- \gamma _i (\mu_p)|<\theta
\end{equation} 
for every $i=1,\ldots ,d$.
\end{theorem}

We point out that the previous theorem remains valid even for discontinuous cocycles, provided they admit continuous holonomies on a suitable open set (see Theorem \ref{theo: general} below). We have chosen to first state and prove the continuous version because its statement is more natural and allows us to avoid unnecessary technicalities.

\section{Proof of Theorem \ref{theo: main}}

In this section, we present the proof of Theorem \ref{theo: main}. The idea is to compare the cocycle along a long recurrent orbit of a regular point $x$ with the cocycle along a nearby periodic orbit $p$ obtained from the periodic closing property. Using the invariant stable and unstable holonomies, we construct a linear map that shows that $A^n(p)$ is similar to a small perturbation of $A^n(x)$. After writing this perturbation in coordinates adapted to the Oseledets splitting, we build a nearby invariant flag by means of the graph transform. This allows us to study the induced quotient operators separately. Each quotient operator is a small perturbation of one Oseledets block, so its eigenvalues remain close to the corresponding Lyapunov exponent. Since the quotient operators together account for all eigenvalues, with the correct multiplicities, the Lyapunov exponents of the periodic orbit are close to those of the original ergodic measure.

To simplify notation, since the ergodic measure $\mu$ is fixed for the entire proof, we simply write $\lambda_j$ instead of $\lambda_j(\mu)$. Moreover, in what follows we assume that $\ell \ge 2$, that is, $(A,f,\mu)$ has at least two different Lyapunov exponents. The case when $\ell = 1$ can be treated using a simplified version of the argument presented below. Furthermore, we will use the letter $C$ to denote a positive constant that may differ in each step.

\subsection{Setting up the proof}\label{sec: setting up the proof}
Let $\Delta:=\min_{j=1,\ldots,\ell-1}\{\lambda_j-\lambda_{j+1}\}$ and take $\theta>0$ such that $4\theta<\Delta$. Given $K>0$, let $\Reg ^{K,\theta}$ be the set of all $x\in \Reg$ such that 
\begin{equation}\label{eq: estimate in Reg}
K^{-1}e^{(\lambda_j-\theta)n}\le \|A^n(x)v\|\le Ke^{(\lambda_j+\theta)n} 
\end{equation}
for every $v\in E^j_x$ with $\|v\|=1$ and $n\in \mathbb{N}$. Note that $\mu(\Reg ^{K,\theta})\to 1$ as $K\to +\infty$. In particular, taking $K$ large enough, by Lusin's Theorem, there exists a compact set $\widetilde \Reg ^{K,\theta}\subset \Reg ^{K,\theta}$ with $\mu\left(\widetilde \Reg ^{K,\theta}\right)>0$ such that the Oseledets splitting is continuous when restricted to it.

Let $\varepsilon,\tau>0$ and $\varepsilon_0,\delta >0$ be the constants given by the local product structure and the periodic closing property, respectively, with $\delta <\tau$. By the Poincaré Recurrence Theorem, there exists $x\in \widetilde \Reg ^{K,\theta}$ and an arbitrarily large $n\in \N$ with 
\begin{equation}\label{eq: size of n}
    K^2\le e^{(\Delta -3\theta)n}
\end{equation}
such that $f^n(x)\in \widetilde \Reg ^{K,\theta}$ and $d(x,f^n(x))<\varepsilon_0$.
Then, by the periodic closing property, there exists a periodic point $p\in M$ with $f^n(p)=p$ such that 
 \[d(f^i(x),f^i(p))\le \delta \quad \text{for every } i=0,1,2,\ldots,n. \]
Thus, since $\delta<\tau$, by the local product structure we know that for each $i=0,1,2,\ldots,n$, $W^s_\varepsilon(f^i(p))$ and $W^u_\varepsilon(f^i(x))$ intersect in a unique point which we denote by
\[z_i:=[f^i(p),f^i(x)]=W^s_\varepsilon(f^i(p))\cap W^u_\varepsilon(f^i(x)).\]

\begin{lemma}
For every $\delta=\delta(f,\varepsilon)>0$ small enough, we have that $f(z_i)=z_{i+1}$ for every $i=0,1,2,\ldots,n-1$.
\end{lemma}
\begin{proof}
From the definition of the stable set and the fact that $z_i\in W^s_\varepsilon(f^i(p))$, we have $f(z_i)\in W^s_\varepsilon(f^{i+1}(p))$ for all $i=0,1,2,\ldots,n-1$. Now, from $z_i \in W^u_\varepsilon(f^{i}(x))$, we obtain $d(f^n(f^{i+1}(x)),f^n(f(z_i)))\leq \varepsilon$ for all $n\leq -1$. Hence, in order to conclude that $f(z_i)\in W^u_\varepsilon(f^{i+1}(x))$, we must only show that the last inequality also holds for $n=0$, that is, $d(f^{i+1}(x),f(z_i))\leq \varepsilon$.

By the uniform continuity of $f$, let $\eta>0$ be such that $d(y_1,y_2)<\eta$ implies $d(f(y_1),f(y_2))<\varepsilon$ for every $y_1,y_2\in M$. Because the map $y\mapsto[y,f^i(x)]$ is uniformly continuous for every $i=0,1,2,\ldots,n-1$, there exists $\delta>0$ such that $d(f^i(x),z_i)=d([f^i(x),f^i(x)],[f^i(p),f^i(x)])<\eta$ whenever $d(f^i(x),f^i(p))<\delta$. Thus, $d(f^{i+1}(x),f(z_i))<\varepsilon$, as desired.
\end{proof}

\subsection{Transition operator} 
We now define a transition operator between the fibers over the orbit of $x$ and the periodic orbit $p$. More precisely, for each $i=0,1,2,\ldots,n$, we define the map $H_{f^i(x) ,f^i(p)}\colon \R^d \to \R^d$ by 
\[H_{f^i(x), f^i(p)} := H^s_{z_i ,f^i(p)}  H^u_{f^i(x) , z_i}.\]

\begin{lemma}\label{lem: equivariance transition map}
For every $i=0,1,2,\ldots,n-1$,
    \[H_{f^{i+1}(x), f^{i+1}(p)}  A(f^i(x)) = A(f^i(p))H_{f^{i}(x), f^{i}(p)} .\]
In particular,
 \begin{equation}\label{eq: aux lemma}
      A^n(p)=H_{f^n(x),p}A^n(x)H_{x,p}^{-1}.
 \end{equation}
\end{lemma}
\begin{proof}
Using the invariance condition (iii) from the definition of holonomies and the fact that $f(z_i)=z_{i+1}$, we obtain, for any $i=0,1,\ldots,n-1$,
\[ \begin{split}
	H_{f^{i+1}(x),f^{i+1}(p)}A(f^{i}(x)) &= H_{f(z_i),f^{i+1}(p)}^{s}H_{f^{i+1}(x),f(z_i)}^{u}A(f^{i}(x)) \\
	&= H_{f(z_i),f^{i+1}(p)}^{s}A(z_i)H_{f^{i}(x),z_i}^{u} \\
	&= A(f^{i}(p))H_{z_i,f^{i}(p)}^{s}H_{f^{i}(x),z_i}^{u} \\
	&= A(f^i(p))H_{f^{i}(x),f^{i}(p)}.
\end{split} \] 
Finally, \eqref{eq: aux lemma} follows directly from the above identity via induction.
\end{proof}

By Lemma \ref{lem: equivariance transition map},
\[A^n(p)=H_{x,p} \left( H_{x,p}^{-1}H_{f^n(x),p}A^n(x) \right) H_{x,p}^{-1}.\]
Thus, since
\[\begin{split}
    H_{x,p}^{-1}H_{f^n(x),p}&= H^u_{z_0, x}  H^s_{p , z_0}  H^s_{z_n , p}  H^u_{f^n(x) , z_n},
\end{split}\]
defining
\begin{equation}\label{eq: def H and Bn}
H_{f^n(x),x}:=H^u_{z_0, x}  H^s_{p , z_0}  H^s_{z_n , p}  H^u_{f^n(x) , z_n} \quad \text{ and }\quad B_n(x):=H_{f^n(x),x}A^n(x),
\end{equation}
we have that
\begin{equation}\label{eq: relat Anp Bnx}
A^n(p)=H_{x,p} B_n(x) H_{x,p}^{-1}.
\end{equation}
That is, the matrices $A^n(p)$ and $B_n(x)$ are similar. In particular,
\[\text{eigenvalues}\left(A^n(p)\right)=\text{eigenvalues}\left(B_n(x)\right).\]
Moreover, observe that $d(f^n(x),x)\leq 2\delta$. Thus, recalling \eqref{eq: def H and Bn}, since the distances between $f^i(x), f^i(p)$ and $z_i$ go uniformly to zero as $\delta\to 0$, $(z,y)\mapsto H^\ast_{z,y}$ is uniformly continuous, and $H^\ast_{x,x}=\Id$ for $\ast=s,u$, we have that 
\begin{equation}\label{eq: Hfnx x close to Id}
    \|H_{f^n(x),x}-\Id\|<\omega_H(\delta)
\end{equation}
where $\omega_H$ is an increasing function satisfying $\omega_H(\delta)\to 0$ as $\delta\to 0$.

\subsection{Block diagonalization and perturbation bounds} \label{sec: block diag}
In this section, we construct a continuous change of coordinates using the Oseledets splitting so that, in this new coordinate system, $A^n(x)$ is a block-diagonal matrix. Then, we construct a matrix which is similar to $A^n(p)$ and is close to the block-diagonal matrix previously obtained.

Fix a recurrent point $x\in\widetilde \Reg^{K,\theta}$ and the corresponding periodic point $p$ of period $n$ obtained in the previous subsection. Over the invariant set $\Reg$, the Oseledets splitting is well-defined. Therefore, we can introduce a measurable family of linear isomorphisms 
\[
\Psi_y\colon \R^d\to \R^d, \qquad y\in \Reg,
\]
such that
\[
\Psi_y(\R^{m_j})=E_y^j
\]
for every $j$, where $m_j=\dim E_y^j$. Here we think of $\R ^d$ as $\R^d=\R^{m_1}\oplus \R^{m_2}\oplus \ldots\oplus \R^{m_\ell}$ and each $\R^{m_j}$ as a subspace of $\R^d$. Since the Oseledets splitting is continuous on the compact subset $\widetilde \Reg ^{K,\theta}$, we can construct this family such that the restriction $y \to \Psi_y$ is continuous on $\widetilde \Reg ^{K,\theta}$. By compactness,
\begin{equation}\label{eq: bound on Psi}
\sup_{y\in\widetilde \Reg ^{K,\theta}} \max\{\|\Psi_y\|,\|\Psi_y^{-1}\|\} =:C_0<\infty .
\end{equation}

 For $y \in \Reg$, define $\widehat A(y)=\Psi_{f(y)}^{-1}A(y)\Psi_y$. Then, evaluating this along the orbit of $x$ gives
\[
\widehat A^n(x)= \Psi_{f^n(x)}^{-1}A^n(x)\Psi_x = \operatorname{diag}(L_1,\ldots,L_\ell),
\]
where each block $L_j$ is an $m_j\times m_j$ matrix. By \eqref{eq: estimate in Reg} and \eqref{eq: bound on Psi} we get that
\begin{equation}\label{eq:block-estimates}
(C_0^2K)^{-1}e^{(\lambda_j-\theta)n} \le m(L_j) \le \|L_j\| \le C_0^2K\,e^{(\lambda_j+\theta)n}
\end{equation}
where $m(L_j)$ denotes the \emph{conorm} of the matrix $L_j$.

Set $\widehat B_n(x) = \Psi_x^{-1}B_n(x)\Psi_x$. Thus, since $B_n(x)=H_{f^n(x),x}A^n(x)$ and $\widehat A^n(x) = \Psi_{f^n(x)}^{-1}A^n(x)\Psi_x$, we can write
\[ \widehat B_n(x) = \Psi_x^{-1}H_{f^n(x),x}\Psi_{f^n(x)}\widehat A^n(x) = (\Id+R_n)\widehat A^n(x), \]
where $\Id + R_n = \Psi_x^{-1}H_{f^n(x),x}\Psi_{f^n(x)}$. We will now estimate the size of $\|R_n\|$. For this purpose, we rewrite it as
\[ R_n = \Psi_x^{-1} \left[ (H_{f^n(x),x} - \Id)\Psi_{f^n(x)} + (\Psi_{f^n(x)} - \Psi_x) \right]. \]
Thus, since $\widetilde \Reg ^{K,\theta}$ is compact and $\widetilde \Reg ^{K,\theta} \ni  y\to \Psi_y$ is continuous, 
\[\|\Psi_w-\Psi_y\|\leq \omega_\Psi(d(w,y))\]
where $\omega_\Psi$ is an increasing function satisfying $\omega_\Psi(\delta)\to 0$ as $\delta\to 0$. Moreover, as
\[d(x,f^n(x))\leq d(x,p)+d(f^n(p),f^n(x))\leq 2\delta,\]
we get that $\|\Psi_{f^n(x)} - \Psi_x\| \le \omega_\Psi(2\delta)$. Combining this fact with \eqref{eq: Hfnx x close to Id} and \eqref{eq: bound on Psi} we get that
\begin{equation}\label{eq: estimate Rn}
\begin{aligned}
\|R_n\| &\le \|\Psi_x^{-1}\| \left( \|H_{f^n(x), x} - \Id\| \|\Psi_{f^n(x)}\| + \|\Psi_{f^n(x)} - \Psi_x\| \right) \\
&\le C_0 \left( \omega_H(\delta) C_0 + \omega_\Psi(2\delta) \right) = C_0^2 \omega_H(\delta) + C_0 \omega_\Psi(2\delta). 
\end{aligned}
\end{equation}
In particular, by choosing $\delta$ sufficiently small, we can make $\|R_n\|$ arbitrarily small.

Finally, we note that, since $A^n(p)$ is similar to $B_n(x)$ (recall \eqref{eq: relat Anp Bnx}) and $B_n(x)$ is similar to $\widehat B_n(x)$, the matrices $A^n(p)$ and $\widehat B_n(x)$ have the same eigenvalues.

\subsection{Invariant flag and quotient dynamics}
In this section, we establish the existence of a nested sequence of $\widehat B_n(x)$-invariant subspaces forming a flag via a graph transform argument. We then analyze the induced quotient dynamics on these nested spaces, conjugating them to explicit linear operators on $\mathbb R^{m_j}$ to derive norm estimates on their perturbation terms.

Let
\[
U_j=\R^{m_1}\oplus\cdots\oplus\R^{m_j}, \qquad V_j=\R^{m_{j+1}}\oplus\cdots\oplus\R^{m_\ell},
\]
so that $\R^d=U_j\oplus V_j$. Relative to this decomposition, we write
\[
\widehat A^n(x)= 
\begin{pmatrix} 
L_{\le j} & 0\\ 
0 & L_{>j} 
\end{pmatrix},
\]
where $L_{\le j}=\operatorname{diag}(L_1,\ldots,L_j)$ and $L_{>j}=\operatorname{diag}(L_{j+1},\ldots,L_\ell)$ and
\[
R_n= 
\begin{pmatrix} 
R_{11} & R_{12}\\ 
R_{21} & R_{22} 
\end{pmatrix},
\]
so that $\widehat B_n(x) = (\Id+R_n) \operatorname{diag}(L_{\le j}, L_{>j})$. Noting that $L_{>j}$ and $L_{\le j}^{-1}$ act on separate coordinate spaces ($V_j$ and $U_j$, respectively), we may bound the norms of the individual blocks using \eqref{eq:block-estimates} by
\begin{equation}\label{eq: est L>jL<j}
\|L_{>j}\| \|L_{\le j}^{-1}\| \le C_0^4K^2 e^{-(\lambda_j-\lambda_{j+1}-2\theta)n} \le C_0^4e^{-\theta n},
\end{equation}
where the last inequality follows from the choice of the return time $n$ in \eqref{eq: size of n}. In particular, $\|L_{>j}\| \|L_{\le j}^{-1}\| \to 0$ as $n\to\infty$ for $j=1,\ldots,\ell $. Consequently, combining this fact with \eqref{eq: estimate Rn}, it follows that after fixing $\delta>0$ sufficiently small and taking the recurrence time $n$ sufficiently large, we may assume without loss of generality that $\|R_n\|+\|L_{>j}\| \|L_{\le j}^{-1}\|$ is as small as we need. This will be used extensively in the sequel.

\begin{lemma}\label{lem:graph-transform}
Assume that $\|R_n\|+\|L_{>j}\| \|L_{\le j}^{-1}\|$ is sufficiently small. Then, for every $j=1,\ldots,\ell-1$, there exists a unique linear map $P_j\colon U_j\to V_j$ with $\|P_j\|\leq1$ such that $F_j=\operatorname{graph}(P_j)$ is invariant under $\widehat B_n(x)$. Moreover, there exists a constant $C>0$, independent of $n$, such that $\|P_j\|\le C\|R_n\|.$
\end{lemma}

\begin{proof}
In what follows, we write $\widehat B_n(x)=(\Id+R_n) \operatorname{diag}(L_{\le j}, L_{>j})$ relative to the decomposition $\R^d=U_j\oplus V_j$. Given a linear operator $P\colon U_j\to V_j$, a vector in $\operatorname{graph}(P)$ has the form $(u,Pu)$ with $u\in U_j$. Thus, its image under $\widehat B_n(x)$ is $(u',v')\in U_j\oplus V_j$ with
\[
u' = (\Id+R_{11})L_{\le j}u + R_{12}L_{>j}Pu \;\text{ and }\; v' = R_{21}L_{\le j}u + (\Id+R_{22})L_{>j}Pu
\]
where $\Id$ denotes the identity operator on the respective space. Hence, $\operatorname{graph}(P)$ is invariant under $\widehat B_n(x)$ if and only if $v'=Pu'$. Since this identity must hold for every $u\in U_j$, we get that
\[
P(\Id+R_{11})L_{\le j} + PR_{12}L_{>j}P = R_{21}L_{\le j} + (\Id+R_{22})L_{>j}P.
\]
Multiplying both sides of this equation on the right by $L_{\le j}^{-1}$ yields the equivalent fixed-point equation $P=\mathcal G_j(P)$, where
\[
\mathcal G_j(P) = \left( R_{21} + (\Id+R_{22}) L_{>j}PL_{\le j}^{-1} \right) \left( \Id + R_{11} + R_{12}L_{>j}PL_{\le j}^{-1} \right)^{-1}.
\]
To simplify notation, let $X(P) = R_{21} + (\Id+R_{22}) L_{>j}PL_{\le j}^{-1}$ and $Y(P) = \Id + R_{11} + R_{12}L_{>j}PL_{\le j}^{-1}$. 

We will now show that $\mathcal G_j$ is well-defined and, moreover, is a contraction on the closed ball $\mathcal B = \{P \in \mathcal{L}(U_j, V_j) : \|P\| \le 1\}$ where $\mathcal{L}(U_j, V_j)$ denotes the space of linear operators from $U_j$ to $V_j$. In fact, since $\|R_n\|$ and $\|L_{>j}\| \|L_{\le j}^{-1}\|$ can be taken arbitrarily small, for any $P \in \mathcal B$, the operator $Y(P)$ is invertible via the Neumann series and $\|Y(P)^{-1}\| \le 2$ and $\|X(P)\| \le 1$. In particular, $\mathcal{G}_j$ is well-defined. Moreover, for any two linear maps $P_1, P_2 \in \mathcal B$, we have that
\[
\mathcal{G}_j(P_1) - \mathcal{G}_j(P_2) = (X(P_1) - X(P_2))Y(P_1)^{-1} + X(P_2) Y(P_1)^{-1} (Y(P_2) - Y(P_1)) Y(P_2)^{-1}.
\]
Now,
\[
X(P_1) - X(P_2) = (\Id+R_{22})L_{>j}(P_1 - P_2)L_{\le j}^{-1}
\]
and
\[
Y(P_2) - Y(P_1) = R_{12}L_{>j}(P_2 - P_1)L_{\le j}^{-1}.
\]
Thus, taking norms we get that
\[
\|\mathcal G_j(P_1)-\mathcal G_j(P_2)\| \le C_1 \|L_{>j}\| \|L_{\le j}^{-1}\| \|P_1-P_2\|,
\]
where $C_1$ is independent of $n$. By choosing the recurrence time $n$ sufficiently large so that $C_1\|L_{>j}\| \|L_{\le j}^{-1}\| \le \frac{1}{2}$, it follows that
\[
\|\mathcal G_j(P_1)-\mathcal G_j(P_2)\| \le \frac{1}{2}\|P_1-P_2\|.
\]
Furthermore, since $\|\mathcal{G}_j(0)\| = \|R_{21}(\Id+R_{11})^{-1}\| \le 2 \|R_n\|\leq 1/2$ for $\delta$ small, we get that for any $P \in \mathcal B$, 
\[
\|\mathcal{G}_j(P)\| \le \|\mathcal{G}_j(0)\| + \frac{1}{2}\|P\| \le \frac{1}{2} + \frac{1}{2} = 1,
\]
which proves $\mathcal{G}_j(\mathcal B) \subset \mathcal B$. Thus, by the Banach Fixed Point Theorem, there exists a unique $P_j \in \mathcal B$ such that $\mathcal G _j(P_j)=P_j$. Finally, the contractive bound for $\mathcal G_j$ yields $\|P_j\| \le 2\|\mathcal{G}_j(0)\|  \le 4\|R_{21}\| \le C\|R_n\|$, where $C$ is independent of $n$, concluding the proof of the lemma.
\end{proof}

Recall that $\widehat A^n(x)$ is a block-diagonal matrix. On the other hand, $\widehat B_n(x)$, which is our matrix of interest, is not due to the perturbation coming from $R_n$. Nevertheless, we show in the next lemma that $\widehat B_n(x)$ preserves a flag, replacing the block decomposition.

\begin{lemma}\label{lem:invariant-flag}
Consider $\widehat B_n(x)=(\Id+R_n)\operatorname{diag}(L_1,\ldots,L_\ell)$ with $\|R_n\|+\|L_{>j}\| \|L_{\le j}^{-1}\|$ sufficiently small so that Lemma \ref{lem:graph-transform} holds for $j=1,2,\ldots,\ell-1$. Then the invariant graphs $F_j=\operatorname{graph}(P_j)$ given by Lemma \ref{lem:graph-transform} satisfy
\[
F_0\subset F_1\subset\cdots\subset F_\ell=\mathbb R^d,
\]
where $F_0=\{0\}$. Consequently, they form a $\widehat B_n(x)$-invariant flag.
\end{lemma}

\begin{proof}
Fix $j\in \{2,\ldots,\ell-1\}$ and let $F_j=\operatorname{graph}(P_j)$ be the invariant graph over $U_j$ given by Lemma \ref{lem:graph-transform}. We will prove that $F_{j-1}\subset F_j$.  Let $\pi_j \colon \mathbb R^d \to U_j$ be the standard coordinate projection and $\Gamma_j \colon U_j \to F_j$ be the graphing isomorphism given by $\Gamma_j(u) = (u, P_j(u))$, so that the restriction $\pi_j|_{F_j} \colon F_j \to U_j$ acts as its inverse. Using these maps, let us define the operator $A_j \colon  U_j \to U_j$ by
\[
A_j := \pi_j|_{F_j} \circ \widehat B_n (x)\circ \Gamma_j.
\]
Note that we have not explicitly stated the dependency of $A_j$ on $x$ and $n$. This shall cause no confusion. Thus, writing $\widehat B_n(x)=(\Id+R_n) \operatorname{diag}(L_{\le j}, L_{>j})$ relative to $\R^d = U_j \oplus V_j$,  we get that for any $u\in U_j$, 
\[
\begin{split}
A_j(u) &= \pi_j \left( \begin{pmatrix} \Id + R_{11} & R_{12} \\ R_{21} & \Id + R_{22} \end{pmatrix} \begin{pmatrix} L_{\le j}u \\ L_{>j}P_j(u) \end{pmatrix} \right)\\
&= (\Id + R_{11})L_{\le j}u + R_{12}L_{>j}P_j(u).
\end{split}
\]
In particular, we can write
\[
A_j = (\Id + S^{(j)})L_{\le j},
\]
where $S^{(j)} = R_{11} + R_{12}L_{>j}P_j L_{\le j}^{-1}$. Moreover, recalling \eqref{eq: est L>jL<j} and using that $\|R_{11}\|, \|R_{12}\| \le \|R_n\|$ and $\|P_j\| \le C\|R_n\|$ where this last inequality comes from Lemma \ref{lem:graph-transform}, by taking $n$ sufficiently large, we get that $\|S^{(j)}\| \le 2\|R_n\|$.

We now consider $A_j$ with respect to the splitting $U_j = U_{j-1}\oplus \R^{m_j}$. Relative to this decomposition,
\[
L_{\le j} = \begin{pmatrix} L_{\le j-1} & 0\\ 0 & L_j \end{pmatrix},
\]
and $\|L_j\| \|L_{\le j-1}^{-1}\| \le \|L_{>j-1}\| \|L_{\le j-1}^{-1}\| $. Moreover, as already observed, $\|S^{(j)}\|\le 2\|R_n\|$. In particular, $\|S^{(j)}\|+\|L_j\| \|L_{\le j-1}^{-1}\|$ may be taken arbitrarily small. Therefore all the hypotheses of Lemma~\ref{lem:graph-transform} are satisfied for the operator $A_j$ acting on the decomposition $U_j=U_{j-1}\oplus\R^{m_j}$. Applying that lemma, there exists a unique $A_j$-invariant graph $G_{j-1} = \operatorname{graph}(Q) \subset U_j$, where $Q\colon U_{j-1}\rightarrow\R^{m_j}$ satisfies $\|Q\| \le C'\|S^{(j)}\| \le 2C'\|R_n\|.$

Now, we lift this sub-graph back to the global space $\mathbb R^d$ by setting $\widetilde F_{j-1} = \Gamma_j(G_{j-1}) \subset F_j$. We claim that $\widetilde F_{j-1}$ is invariant under $\widehat B_n(x)$. Indeed, given $w \in \widetilde F_{j-1}$, there exists $u \in G_{j-1}$ such that $w = \Gamma_j(u) $. Thus, since $\Gamma_j(u) \in F_j$ and $F_j$ is $\widehat B_n(x)$-invariant, the vector $\widehat B_n(x)(\Gamma_j(u))$ must lie entirely in $F_j$. Consequently, as $\Gamma_j \circ \pi_j$ acts as the identity map on $F_j$, we have
\[
\widehat B_n(x)(w) = \widehat B_n(x)(\Gamma_j(u)) = \Gamma_j \left( \pi_j(\widehat B_n(\Gamma_j(u))) \right) = \Gamma_j(A_j(u)).
\]
Thus, since $G_{j-1}$ is $A_j$-invariant, it follows that $A_j(u) \in G_{j-1}$, which implies $\Gamma_j(A_j(u)) \in \Gamma_j(G_{j-1}) = \widetilde F_{j-1}$, proving that $\widetilde F_{j-1}$ is $\widehat B_n(x)$-invariant.

Moreover, from the definition of $\widetilde F_{j-1}$, it follows that every vector in $\widetilde F_{j-1}$ can be expressed as $u + Q(u) + P_j(u+Q(u))$ for $u \in U_{j-1}$. Since $Q(u) \in \mathbb R^{m_j} \subset V_{j-1}$ and $P_j(u+Q(u)) \in V_j \subset V_{j-1}$, this subspace can be rewritten directly as a global graph over $U_{j-1}$. Namely,
\[
\widetilde F_{j-1} = \operatorname{graph}(\widetilde P), 
\]
where $\widetilde P:U_{j-1}\to V_{j-1}$ is given by $\widetilde P(u) = Q(u) + P_j(u+Q(u))$. Furthermore, the norm estimates on $Q$ and $P_j$ imply $\|\widetilde P\| \le C''\|R_n\|$ for some $C''$ which is independent on $n$. Thus, by taking $n$ sufficiently large and $\delta$ sufficiently small so that $C''\|R_n\|<1/2$, we get that $\|\widetilde P\|\leq 1$. In particular, $\widetilde P$ is inside the ball where the graph transform operator $\mathcal{G}_{j-1}$ from Lemma~\ref{lem:graph-transform} acts as a contraction.  Moreover, since both $\widetilde F_{j-1} = \operatorname{graph}(\widetilde P)$ and $F_{j-1} = \operatorname{graph}(P_{j-1})$ are $\widehat B_n(x)$-invariant graphs over $U_{j-1}$, it follows from the construction of $\mathcal{G}_{j-1}$ that they must be fixed points of this operator. Thus, since the fixed point of a contraction is unique, we get that $\widetilde P = P_{j-1}$, and thus $\widetilde F_{j-1} = F_{j-1}$. Finally, recalling that, by construction, $\widetilde F_{j-1} \subset F_j$, it follows that $F_{j-1} \subset F_j$, completing the proof.
\end{proof}

\begin{lemma}\label{lem:quotient-representation}
For each $j=1,\ldots,\ell$, let $\widetilde T_j\colon F_j/F_{j-1} \to F_j/F_{j-1}$ denote the quotient operator induced by $\widehat B_n(x)$. Then $\widetilde T_j$ is naturally conjugated to a linear operator $T_j\colon \R^{m_j}\to\R^{m_j}$ of the form
\begin{equation}\label{eq:multiplicative-block-form}
T_j=(\Id+S_j)L_j,
\end{equation}
where
\[
\|S_j\| \le C\left( C_0^2\omega_H(\delta) + C_0\omega_\Psi(2\delta) \right),
\]
for some constant $C>0$ independent of $n$.
\end{lemma}

\begin{proof}
We retain all the notation from the proof of Lemma~\ref{lem:invariant-flag}. We have established there that the linear graph parametrization $\Gamma_j \colon U_j\to F_j$ satisfies $\Gamma_j(G_{j-1})=F_{j-1}$. Consequently, $\Gamma_j$ induces a natural quotient isomorphism
\[
\overline\Gamma_j\colon U_j/G_{j-1} \to F_j/F_{j-1},
\]
given by $\overline\Gamma_j([u])=[\Gamma_j(u)]$. Moreover, since $G_{j-1}$ is $A_j$-invariant, the quotient operator $\widetilde A_j\colon U_j/G_{j-1} \to U_j/G_{j-1}$ given by $\widetilde A_j([u])=[A_j(u)]$
is well-defined.

On the other hand, since $F_{j-1}$ is $\widehat B_n(x)$-invariant, the restriction of $\widehat B_n(x)$ to $F_j$ induces the quotient operator
\[
\widetilde T_j\colon F_j/F_{j-1} \to F_j/F_{j-1},
\]
appearing in the statement of the lemma, namely, $\widetilde T_j([v])=[\widehat B_n(x)v]$ for $v\in F_j$.
Thus, since $\Gamma_jA_j=\widehat B_n(x)|_{F_j}\Gamma_j$, passing to the quotient gives us that
\[
\widetilde T_j = \overline\Gamma_j \widetilde A_j \overline\Gamma_j^{-1}.
\]

Now, we observe  that, since $G_{j-1} = \operatorname{graph}(Q)$ is an $A_j$-invariant subspace of $U_j$ where $Q\colon U_{j-1} \to \mathbb R^{m_j}$, we have the direct sum decomposition $U_j = G_{j-1} \oplus \mathbb R^{m_j}$. 
Indeed, every vector $(x,y)\in U_{j-1}\oplus\mathbb R^{m_j}=U_j$ can be written uniquely as
\[
(x,y)=(x,Qx)+(0,y-Qx),
\]
where $(x,Qx)\in G_{j-1}$ and $(0,y-Qx)\in\mathbb R^{m_j}$. 
Hence the quotient space $U_j/G_{j-1}$ is naturally identified with
$\mathbb R^{m_j}$ through the linear isomorphism
\[
\Phi\colon U_j/G_{j-1}\to \mathbb R^{m_j}\]
given by $[u]\mapsto w$ where $w$ is the unique representative of the class $[u]$ contained in $\mathbb R^{m_j}$. Using this identification, we define $T_j\colon \mathbb R^{m_j}\to \mathbb R^{m_j}$ by
\[
T_j=\Phi\widetilde A_j\Phi^{-1}.
\]
In particular,
\[
T_j=\Phi\overline\Gamma_j^{-1}\widetilde T_j\overline\Gamma_j\Phi^{-1},
\]
so that $T_j$ is naturally conjugated to the quotient operator $\widetilde T_j$.

Let $J\colon \mathbb R^{m_j}\hookrightarrow U_j$ denote the natural inclusion and $\Pi^Q\colon U_j\to \mathbb R^{m_j}$ be the projection parallel to $G_{j-1}$. Since
\[
\Phi^{-1}(w)=[Jw] \;\text{ and }\;\Phi([u])=\Pi^Q(u),
\]
we get that
\[
T_j=\Pi^Q A_j J.
\]

By Lemma~\ref{lem:invariant-flag}, $A_j=(\Id+S^{(j)})L_{\le j}$ with $\|S^{(j)}\|\le C\|R_n\|$. Thus, since $L_{\le j}J=JL_j$ and $\Pi^QJ=\Id$,
it follows that
\[
\begin{aligned}
T_j&=\Pi^Q(\Id+S^{(j)})L_{\le j}J=(\Pi^QJ+\Pi^QS^{(j)}J)L_j=(\Id+S_j)L_j,
\end{aligned}
\]
where $S_j=\Pi^QS^{(j)}J$. Moreover, since $\|Q\|\leq 2C'\|R_n\|$, the angle between $G_{j-1}$ and $\mathbb R^{m_j}$ is uniformly bounded away from zero. In particular, $\|\Pi^Q\|\le C$ for some constant $C>0$. Therefore,
\[
\|S_j\|\le\|\Pi^Q\|\|S^{(j)}\|\|J\|\le C\|R_n\|.
\]
Combining this estimate with \eqref{eq: estimate Rn} we conclude that
\[
\|S_j\|\le C\left(C_0^2\omega_H(\delta)+ C_0\omega_\Psi(2\delta)\right),
\]
which completes the proof.
\end{proof}

\subsection{Localization and bounds for eigenvalues}
In this section, we establish a relationship between the eigenvalues of a linear map and those of its induced quotient operator. Furthermore, we provide explicit bounds for the eigenvalues of certain perturbations of a linear map. These are all simple observations which hold for general linear maps.

\begin{lemma}\label{lem:spectrum-flag}
Let $0=F_0\subset F_1\subset\cdots\subset F_\ell=\mathbb R^d$ be an invariant flag for a linear operator $T\colon \mathbb R^d\to \mathbb R^d$. For each $j$, let $\widetilde T_j \colon F_j/F_{j-1}\to F_j/F_{j-1}$ denote the induced quotient operator. Then
\[
\sigma(T) = \bigcup_{j=1}^{\ell}\sigma(\widetilde T_j),
\]
where $\sigma(\cdot)$ denotes the spectrum of the operator and the union is taken counting algebraic multiplicities.
\end{lemma}

\begin{proof}
Let us consider a basis of $\mathbb R^d$ adapted to the flag. More precisely, choose vectors $\mathcal B_j$ whose images form a basis of $F_j/F_{j-1}$, and concatenate them to get a basis for $\mathbb R^d$. Relative to the resulting basis, the matrix of $T$ is block upper triangular,
\[
[T]=
\begin{pmatrix}
T_1&*&\cdots&*\\
0&T_2&\cdots&*\\
\vdots&\ddots&\ddots&\vdots\\
0&\cdots&0&T_\ell
\end{pmatrix},
\]
where the diagonal block $T_j$ is precisely the matrix of the induced quotient operator. Therefore,
\[
\det(\lambda \Id-T) = \prod_{j=1}^{\ell} \det(\lambda \Id-T_j).
\]
Hence, the characteristic polynomial of $T$ is the product of the characteristic polynomials of the quotient operators. The conclusion follows immediately.
\end{proof}

\begin{lemma}\label{lem:bilateral-eigenvalue}
Let $\rho \in \mathbb{C}$ be any eigenvalue of $T_j = (\Id + S_j)L_j$ with $\|S_j\| < 1$. Then $\rho$ satisfies 
\begin{equation*}
(1 - \|S_j\|)m(L_j) \le |\rho| \le (1 + \|S_j\|)\|L_j\|.
\end{equation*}
\end{lemma}

\begin{proof}
We work in the complexification $\mathbb C^{m_j}$, equipped with the Hermitian extension of the Euclidean norm. Let $v \in \mathbb{C}^{m_j}$ be such that $T_j v = \rho v$ with $\|v\| = 1$. By definition of $T_j$, we have $\rho v = (\Id + S_j)L_j v$. Taking norms and applying the triangle inequality we get that
\[|\rho| = \|\rho v\| =\|(\Id + S_j)L_j v\| \le \|\Id + S_j\| \|L_j v\| \le (1 + \|S_j\|)\|L_j\|.\]
Conversely, applying the reverse triangle inequality yields
\[ \begin{split}
    |\rho| = \|\rho v\| 
	\geq \|L_j v\| - \|S_j L_j v\| 
	\geq (1 - \|S_j\|) \|L_j v\| 
	\geq (1 - \|S_j\|) m(L_j),
\end{split} \]
which completes the proof.
\end{proof}

\subsection{Completion of the proof}

We are now ready to finish the proof of Theorem \ref{theo: main}. We start recalling that, since $p$ is a periodic point of period $n$, the Lyapunov exponents of the periodic measure $\mu_p$ are precisely
\[
\frac1n\log|\rho|,
\]
where $\rho$ runs over the eigenvalues of $A^n(p)$, counted according to algebraic multiplicity. Also recall that the eigenvalues of $A^n(p)$ are the same as those of $\widehat B_n(x)$. 

For each fixed $j=1,\ldots,\ell$, let $\rho^j_1, \ldots \rho^j_{m_j}$ be the eigenvalues of $T_j = (\Id + S_j)L_j$, counted with algebraic multiplicity, with $|\rho^j_1| \geq \cdots \geq |\rho^j_{m_j}|$. By Lemma \ref{lem:quotient-representation}, for each $j$, $\sigma(T_j) = \sigma(\widetilde T_j)$, and hence, by Lemma \ref{lem:spectrum-flag}, $\sigma(\widehat B_n(x)) = \bigcup_{j=1}^{\ell}\sigma(T_j)$, counting with algebraic multiplicity. Thus, the quantities $\frac1n\log|\rho^j_k|$, for $j=1,\ldots,\ell$ and $k=1,\ldots,m_j$, are precisely the Lyapunov exponents of $\mu_p$. Moreover, Lemma \ref{lem:bilateral-eigenvalue} together with \eqref{eq:block-estimates} yields
\[
(1 - \|S_j\|)(C_0^2K)^{-1} e^{(\lambda_j-\theta)n} \le |\rho^j_k| \le (1 + \|S_j\|)C_0^2Ke^{(\lambda_j+\theta)n}
\]
for all $j=1,\ldots,\ell$ and $k=1,\ldots,m_j$, from which we get
\[
\frac1n\log(1-\|S_j\|) - \frac1n\log(C_0^2K) + (\lambda_j-\theta) \leq \frac1n\log|\rho^j_k|
\]
as well as
\[
\frac1n\log(1+\|S_j\|) + \frac1n\log(C_0^2K) + (\lambda_j+\theta) \geq \frac1n\log|\rho^j_k|.
\]
Thus, by taking $\delta>0$ sufficiently small and $n$ sufficiently large, recalling the bound on $\|S_j\|$ given by Lemma \ref{lem:quotient-representation}, the previous two relations give us that
\[
\lambda_j-2\theta \leq \frac1n\log|\rho^j_k| \leq \lambda_j + 2\theta,
\]
and, consequently,
\begin{equation}\label{eq:exponent-approx}
\left|\frac1n\log|\rho^j_k| - \lambda_j \right| \leq 2\theta
\end{equation}
for all $j=1,\ldots,\ell$ and $k=1,\ldots,m_j$.

Finally, to complete the proof of the theorem, it remains to show that \eqref{eq:exponent-approx} implies \eqref{eq: main theo} (with $2\theta$ instead of just $\theta$). To this end, we only need to verify that the periodic exponents $\frac1n\log|\rho^j_k|$ are ordered appropriately. Specifically, that
\[
\begin{split}
\frac1n\log|\rho^1_1| \geq \cdots \geq \frac1n\log|\rho^1_{m_1}| \geq \frac1n\log|\rho^2_1| \geq \cdots \geq \frac1n\log|\rho^2_{m_2}| \geq \frac1n\log|\rho^3_1| \geq \cdots \\ \cdots \geq\frac1n\log|\rho^\ell_{m_\ell}|.
\end{split}
\]
To establish this global ordering, it suffices to show that $\frac1n\log|\rho^j_{m_j}| \geq \frac1n\log|\rho^{j+1}_1|$ for all $j=1,\ldots,\ell-1$. Now, combining \eqref{eq:exponent-approx} with the initial assumption $\Delta>4\theta$, we obtain
\[
\begin{split}
\frac1n\log|\rho^j_{m_j}| &\geq \lambda_j-2\theta = (\lambda_j-\lambda_{j+1})+\lambda_{j+1} - 2\theta \\ &\geq \Delta + \lambda_{j+1} - 2\theta > \lambda_{j+1} + 2\theta \\ &\geq \frac1n\log|\rho^{j+1}_1|
\end{split}
\] 
which then completes the proof of Theorem \ref{theo: main}.

\subsection{General statement}

One may observe that the continuity hypothesis on the cocycle $A\colon M\to GL(d)$ in the statement of Theorem \ref{theo: main} is not used anywhere in the proof. In fact, we only needed continuity of the function $(x,y)\mapsto [x,y]$ given by the local product structure and of the stable and unstable holonomies. The latter, however, are generally expected to be continuous only if there is some continuity condition on the cocycle. Nevertheless, a careful reader may have noticed that even the holonomies need not be continuous everywhere, only in a neighborhood of the point $x \in \widetilde \Reg ^{K,\theta}$, which is fixed throughout the proof. Below, we restate Theorem \ref{theo: main} more generally by relaxing such continuity assumptions. In the next section, we provide an example of a discontinuous cocycle that fits the general case.

\begin{theorem}\label{theo: general}
Let $f\colon M\to M $ be a homeomorphism with local product structure exhibiting the periodic closing property, $\mu$ an ergodic $f$-invariant probability measure, and $A\colon M\to GL(d,\mathbb{R})$ a measurable map with $\log^+\|A^{\pm1}\| \in L^1(\mu)$ admitting a family of (not necessarily continuous) invariant stable and unstable holonomies. Assume there exists an open set $W \subseteq M$ such that the restrictions $H^s|_{W \times W}$ and $H^u|_{W \times W}$ are continuous and $\operatorname{supp}(\mu) \cap W \neq \varnothing$. Then, for any $\theta>0$, there exists a periodic point $p\in M$ such that
\begin{equation*}
	|\gamma _i (\mu)- \gamma _i (\mu_p)|<\theta
\end{equation*} 
for every $i=1,\ldots ,d$.
\end{theorem}

\begin{proof}
    In the proof setup of Theorem \ref{theo: main}, we may assume $\mu\left( \widetilde \Reg ^{K,\theta} \right) > 1-\mu(W)$, which then implies the recurrent point $x \in \widetilde \Reg ^{K,\theta}$ may also be chosen to be inside $W$. Thus, for sufficiently large $n$, we also obtain $f^n(x) \in W$. This enables us to obtain \eqref{eq: Hfnx x close to Id} from \eqref{eq: def H and Bn} by observing that the maps $(z,y)\mapsto H^\ast_{z,y}$ ($\ast=s,u$) are uniformly continuous in a compact neighborhood of $(x,x) \in W\times W$, inside of which lie the pairs $(z_0,x)$, $(p,z_0)$, $(z_n,p)$ and $(f^n(x),z_n)$ for sufficiently large $n$. The rest of the proof follows identically.
\end{proof}

\section{Examples}\label{sec: examples}
There are several well-known classes of cocycles admitting a family of continuous invariant holonomies as defined in Section \ref{sec: cocycles and inv hol}. Examples include fiber-bunched cocycles over hyperbolic systems \cite{BGV03,Via08} and locally constant cocycles over subshifts of finite type \cite{BBB}, for which a family of holonomies can be obtained via the formulas
\[H^s_{y,z}=\lim_{n\to +\infty}A^n(z)^{-1}A^n(y) \; \text{ for } y,z\in W^s_\varepsilon(x)\]
and
\[H^u_{y,z}=\lim_{n\to +\infty}A^{-n}(z)^{-1}A^{-n}(y) \; \text{ for } y,z\in W^u_\varepsilon(x).\]
However, these examples are H\"older continuous, and thus earlier versions of the approximation results are already applicable to them. In this section, we present examples of continuous cocycles that are not H\"older continuous but still admit a family of continuous invariant holonomies. In particular, our main result applies to them, whereas previous results do not. We also present an example of a discontinuous cocycle that fits in the setting of Theorem \ref{theo: general}.

\begin{example}\label{example1}
Let $f\colon M\to M$ be an Anosov diffeomorphism, $P\colon M\to GL(d,\mathbb R)$ a continuous map and $B\in GL(d,\mathbb R) $. Consider $A\colon M\to GL(d,\mathbb R)$ given by
    \[A(x)=P(f(x))BP(x)^{-1} \;\text{ for } x\in M.\]
 Then 
\[H^s_{ y,z} = P(z) P(y)^{-1} \quad \text{and} \quad H^u_{y,z} = P(z) P(y)^{-1}\]
 form a family of continuous invariant holonomies for $A$. Moreover, we can choose $B$ and $P$ appropriately so that $A$ is not H\"older continuous.
\end{example}

\begin{example}\label{example2}
Let $M = \{0, 1\}^{\mathbb{Z}}$ be the space of bi-infinite sequences $x = (\dots, x_{-1}, x_0, x_1, \dots)$, equipped with the shift map $f \colon M\to M$ defined by $(f(x))_i = x_{i+1}$. We consider $M$ endowed with the standard metric $d(x,y) = \beta^{N(x,y)}$ for a fixed $0 < \beta < 1$, where $N(x,y) = \min\{|i| : x_i \neq y_i\}$. Recall that the local stable and unstable manifolds for a point $x \in M$ are defined respectively as
\[
W^s_{\text{loc}}(x) = \{y \in M : y_i = x_i, \; \forall i \ge 0\},
\]
and
\[
W^u_{\text{loc}}(x) = \{y \in M : y_i = x_i, \; \forall i \le 0\}.
\]

Let $A \colon M \to \text{GL}(2, \mathbb{R})$ be given by
\[
A(x) = \begin{pmatrix} a(x) & 0 \\ 0 & a(x)^{-1} \end{pmatrix}, 
\]
where the scalar function $\log a(x)$ is defined by
\[
\log a(x) = \sum_{k=0}^{\infty} \frac{x_k}{(k+1)^3}.
\]
Note that $\log a(x)$ depends only on future coordinates. Moreover, $A$ is continuous but not H\"older continuous. Indeed, the $n$-th variation of $\log a$ satisfies
\[
\operatorname{var}_n(\log a) = \sup_{N(x,y) \ge n} |\log a(x) - \log a(y)| = \sum_{k \ge n} \frac{1}{(k+1)^3} \sim \frac{1}{2n^2}.
\]
This holds because $N(x,y) \ge n$ implies $x_k = y_k$ for all $|k|\leq n-1$, and $|x_k - y_k| \le 1$ for every $k\in \mathbb{Z}$. Furthermore, the supremum is attained since there exist sequences $x$ and $y$ with $N(x,y)=n$ such that $x_k - y_k = 1$ for all $k\geq n$. In particular, $\operatorname{var}_n(\log a)$ decays polynomially, meaning $\log a$ is continuous. However, if $\log a$ were H\"older continuous, we would have $\operatorname{var}_n(\log a)\leq C\beta ^{\alpha n}$ for some $C,\alpha >0$, which contradicts the polynomial decay established above. Thus, $\log a$ is not H\"older continuous.

We now show that $A$ admits continuous stable and unstable holonomies. Given $y \in W^s_{\text{loc}}(x)$, we have $y_i = x_i$ for all $i \ge 0$. Since $\log a(x)$ depends only on non-negative coordinates, $a(x) = a(y)$, yielding $A(x) = A(y)$. Consequently, $H^s_{x,y}=\text{Id}$ satisfies properties i)--iii) from Section \ref{sec: cocycles and inv hol}. 

For $y \in W^u_{\text{loc}}(x)$, consider
\[
H^u_{x,y} = \lim_{n \to +\infty} A(f^{-1}(y)) \dots A(f^{-n}(y)) A(f^{-n}(x))^{-1} \dots A(f^{-1}(x))^{-1}.
\]
Since $A$ is diagonal, the holonomy matrix takes the form 
\[
H^u_{x,y} = \begin{pmatrix} e^{h^u(x,y)} & 0 \\ 0 & e^{-h^u(x,y)} \end{pmatrix},
\] 
where 
\[
h^u(x,y) = \sum_{j=1}^{\infty} \left( \log a(f^{-j}(y)) - \log a(f^{-j}(x)) \right).
\]
Moreover, as $y \in W_{\text{loc}}^{u}(x)$, the coordinates of $x$ and $y$ coincide for all non-positive indices, which gives us that $N(x,y) \ge 1$. By the definition of $\log a(x)$ and standard integral bounds for series, we have
\[ \begin{split}
	|\log a(f^{-j}(y)) - \log a(f^{-j}(x))| &\leq \sum_{k=0}^{\infty} \frac{|(f^{-j}(y))_{k} -(f^{-j}(x))_{k}|}{(k+1)^3}  \\
		&= \sum_{k\geq j+N(x,y)} \frac{|y_{k-j} - x_{k-j}|}{(k+1)^3} \leq \frac{1}{2(j+N(x,y))^2}.
\end{split} \]
Summing over $j\geq1$, we obtain
\[ \begin{split}
	\sum_{j=1}^{\infty} |\log a(f^{-j}(y)) -\log a(f^{-j}(x))| 		&\leq \int_{N(x,y)}^{\infty} \frac{1}{2t^2}\;dt = \frac{1}{2 N(x,y)}.
\end{split} \]
 % \[ \begin{split}
 % 	\sum_{j=1}^{\infty} |\log a(f^{-j}(y)) -\log a(f^{-j}(x))| &\leq \sum_{j=1}^{\infty} \frac{1}{2(j+N(x,y))^2}  \\
 % 		&= \sum_{i=1+N(x,y)}^{\infty} \frac{1}{2i^2}  \\
 % 		&\leq \int_{N(x,y)}^{\infty} \frac{1}{2t^2}\;dt = \frac{1}{2 N(x,y)}.
 % \end{split} \]
Thus, by the Weierstrass $M$-test, the series defining $h^u(x,y)$ converges absolutely and uniformly. Consequently, $H^u_{x,y}$ depends continuously on $x$ and $y$. Finally, it follows easily from the definition that $H^u_{x,y}$ satisfies properties i)-iii) from Section \ref{sec: cocycles and inv hol}. This concludes the construction of the example.
\end{example}

\begin{example}
Now we give an example of a discontinuous cocycle that falls in the context of Theorem \ref{theo: general}.
Let $f\colon M\to M$ be the shift on $M=\{0,1\}^\Z$ as in Example \ref{example2}. For any point $x=(x_n)_{n\in\Z} \in M$, let
\[
p(x)=\sum_{n=0}^{\infty}\frac{2x_n}{3^{n+1}}
\; \text{ and }\;
q(x)=\sum_{n=1}^{\infty}\frac{2x_{-n}}{3^n},
\]
consider the function
\[
F(a,b)=
\begin{cases}
\dfrac{a^2}{a^2+b^2},&(a,b)\neq(0,0),\\[1ex]
0,&(a,b)=(0,0),
\end{cases}
\]
and define
\[
P(x)=
\begin{pmatrix}
e^{F(p(x),q(x))}&0\\
0&1
\end{pmatrix}
\; \text{ and }\;
B=
\begin{pmatrix}
e^{\lambda_1}&0\\
0&e^{\lambda_2}
\end{pmatrix}
\]
for arbitrary $\lambda_1, \lambda_2 \in \R$. Then set $A(x)=P(f(x))BP(x)^{-1}$, or explicitly,

\[
A(x)=
\begin{pmatrix}
e^{\lambda_1+F(p(f(x)),q(f(x)))-F(p(x),q(x))}&0\\
0&e^{\lambda_2}
\end{pmatrix}.
\]
To see that $A$ is discontinuous, consider the sequence $x^{(k)}=(x^{(k)}_n)_{n\in\Z}$ given by
\[
x^{(k)}_n=
\begin{cases}
1,&n=k\text{ or }n=-(k+1),\\
0,&\text{otherwise}.
\end{cases}
\]
for every $k\ge1$. Then $x^{(k)}\to 0^\Z$ as $k\to\infty$ and
\[
A(x^{(k)})
=
\begin{pmatrix}
e^{\lambda_1+20/41}&0\\
0&e^{\lambda_2}
\end{pmatrix}
\]
for every $k\ge1$, whereas $A(0^\Z)=B$. Moreover, since $P$ and $P^{-1}$ are uniformly bounded, $\log^+\|A^{\pm1}\| \in L^1(\mu)$ for any ergodic probability measure $\mu$. 

On the other hand, similarly to Example \ref{example1}, $A$ admits stable and unstable holonomies given by $H^\ast_{y,z}=P(z)P(y)^{-1}$ for $\ast=s,u$. Because they are continuous for pairs $(x,y)$ such that
neither $x$ nor $y$ is $0^{\mathbb Z}$, the assumptions of Theorem \ref{theo: general} are satisfied for $W=M\setminus\{0^\Z\}$ with any ergodic probability $\mu$ other than the Dirac mass at $0^\Z$. 
\end{example}

%%%%%%%%%%%%%%%%%%%%%%%%%%%%%%%%%%%%%%%%%%%%%%%%%%%%%%%%%
%%%%%%%%%%%%%%%%%%%%%%%%%%%%%%%%%%%%%%%%%%%%%%%%%%%%%%%%%

\medskip{\bf Acknowledgments.}
L. Backes was partially supported by a CNPq-Brazil PQ fellowship under Grant No. 304806/2024-2. B. R. Lemos was supported by a doctoral fellowship from CNPq-Brazil. B. Rocha was supported by the Coordenação de Aperfeiçoamento de Pessoal de Nível Superior - Brasil (CAPES). This work was also partially supported by FAPERGS - Programa Pesquisador Gaúcho - PqG under Grant No. 25/2551-0002627-0.

% \vspace{0.1in}
% \medskip{\bf Statements and Declarations}
% \vspace{0.1in}

% \textbf{Competing Interests:} no potential conflict of interest was reported by the author.

%%%%%%%%%%%%%%%%%%%%%%%%%%%%%%%%%%%%%%%%%%%%%%%%%%%%%%%%%
%%%%%%%%%%%%%%%%%%%%%%%%%%%%%%%%%%%%%%%%%%%%%%%%%%%%%%%%%

\end{document}